\documentclass[12pt,leqno,a4paper]{amsart}
\usepackage{amssymb}
\newcommand{\FF}{{\mathbb{F}}}
\newcommand{\NN}{{\mathbb{N}}}
\newcommand{\QQ}{{\mathbb{Q}}}
\newcommand{\ZZ}{{\mathbb{Z}}}

\newcommand{\fA}{{\mathfrak{A}}}
\newcommand{\fS}{{\mathfrak{S}}}

\newcommand{\bC}{{\mathbf{C}}}
\newcommand{\bG}{{\mathbf{G}}}
\newcommand{\bL}{{\mathbf{L}}}
\newcommand{\bq}{{\mathbf{q}}}
\newcommand{\bu}{{\mathbf{u}}}

\newcommand{\cE}{{\mathcal{E}}}
\newcommand{\cH}{{\mathcal{H}}}

\newcommand{\Ce}{{\rm C}}
\newcommand{\No}{{\rm N}}

\newcommand{\Deg}{{\operatorname{Deg}}}
\newcommand{\Irr}{{\operatorname{Irr}}}
\newcommand{\rk}{{\operatorname{rk}}}

\newcommand{\GL}{{\operatorname{GL}}}
\newcommand{\SL}{{\operatorname{SL}}}
\newcommand{\SU}{{\operatorname{SU}}}
\newcommand{\Sp}{{\operatorname{Sp}}}

\newcommand{\tw}[1]{{}^#1\!}

\let\eps=\epsilon

\newcommand\oneone{\buildrel{1-1}\over\longrightarrow}

\def\qed{{~~~\vrule height .75em width .4em depth .3em}}
\def\irr#1{{\rm Irr}(#1)}
\def\ibr#1{{\rm IBr}(#1)}

\def\oh#1#2{{\bf O}_{#1}(#2)}
\def\Oh#1#2{{\bf O}^{#1}(#2)}

\def\cent#1#2{{\bf C}_{#1}(#2)}
\def\syl#1#2{{\rm Syl}_#1(#2)}
\def\nor{\triangleleft\,}

\def\norm#1#2{{\bf N}_{#1}(#2)}
\def\iitem#1{\goodbreak\par\noindent{\bf #1}}

\let\phi=\varphi
\def\sbs{\subseteq}

\newtheorem{thm}{Theorem}[section]
\newtheorem{lem}[thm]{Lemma}

\newtheorem{cor}[thm]{Corollary}

\newtheorem{prop}[thm]{Proposition}
\newtheorem{question}[thm]{Question}

\theoremstyle{definition}

\theoremstyle{remark}

\begin{document}

\title{Blocks with Equal Height Zero Degrees}

\date{\today}

\author{Gunter Malle}
\address{FB Mathematik, TU Kaiserslautern,
Postfach 3049, 67653 Kaisers\-lautern, Germany.}
\makeatletter
\email{malle@mathematik.uni-kl.de}
\makeatother
\author{Gabriel Navarro}
\address{Departament d'\`Algebra, Universitat de Val\`encia,
Dr. Moliner 50, 46100 Burjassot, Spain.}

\makeatletter
\email{gabriel.navarro@uv.es}
\makeatother

\thanks{The first
author thanks the Isaac Newton Institute for Mathematical Sciences,
Cambridge, for its hospitality during the preparation of part of this work}

\begin{abstract}
We investigate a natural class
of blocks of finite groups: the blocks such that
all of their height zero characters have the
same degree. It is conceivable that these blocks,
which are globally defined,
are exactly the  Brou\'e-Puig (locally defined) nilpotent blocks 
and we offer some partial results in this direction.
The most difficult result here is to prove
that, with one family of possible exceptions, blocks with equal height
zero degrees of simple groups have abelian defect groups and are in fact
nilpotent. 

\end{abstract}

\maketitle

\pagestyle{myheadings}
\markboth{Gunter Malle and Gabriel Navarro}{Nilpotent blocks}

\section{Introduction} \label{sec:main}

The celebrated nilpotent blocks of finite groups introduced by M. Brou\'e and
L. Puig in 1980 (\cite{B-P}) are locally defined in terms of the Alperin-Brou\'e
subpairs (\cite{A-B}). There is a general consensus that nilpotent blocks are
the most natural blocks from the local point of view. It is not easy, however,
to check if a block is nilpotent or not, and  to have a global characterization
of them, especially one that can be detected in the character table of the
group, would be quite interesting.
\smallskip

Here we propose to study blocks $B$ of a finite group $G$ such that all of
its height zero characters $\chi \in {\rm Irr}_0(B)$ have the same degree $d$.
This property of blocks, that can  easily be detected in the character table
of $G$, seem to appear quite naturally in block theory, and deserves some
consideration.
The blocks {\sl all} of whose irreducible characters have the same degree
were already considered by T. Okuyama and Y. Tsushima in \cite{O-T}.

\smallskip
In a nilpotent block $B$ all height zero degrees are equal. And we suspect
that the converse might be true. In this paper, we are able to prove this
in some cases, with quite different arguments. 

\smallskip

If $B$ is the principal block of $G$, or if the defect group
$D$ of $B$ is normal in $G$,
or if $D$ is abelian (and we assume the Height Zero Conjecture)
then  the blocks with equal height zero character degrees are nilpotent. 
These results
constitute Sections 3, 4, and 5 below. 

\smallskip
The most difficult result in this paper, to which a
large extent of it is devoted, is to prove that 
the blocks of simple groups with equal height zero
degrees have abelian defect groups and satisfy Brauer's Height Zero Conjecture.
By our previously mentioned result, this implies that equal height zero degrees
blocks are also nilpotent. This certainly agrees with the recent work of
J. An and C. Eaton in which they prove that nilpotent blocks
of simple groups have abelian defect groups for $p>2$ \cite{AE09}.

\smallskip
The study of  blocks 
of $p$-solvable groups with equal height zero degrees,
 which we do in the last section of the paper, leads to a variation of
a classical large orbit question
which does not seem easy to solve and which has interest in its
own. (Some recent partial results are given in  \cite{D-N}.)
This new type of orbit problem has  connections
with delicate questions on the $p'$-character degrees of finite groups.

\smallskip

Finally, let us mention that the blocks $B$ such that all character degrees
$\chi(1)$ are $p$-powers for $\chi \in \irr B$ give another example of
blocks with equal height zero characters degrees. These blocks
were proved to be nilpotent by work of G. R. Robinson and the second author
(\cite{N-R}).

\section{EHZD blocks and Nilpotent Blocks} \label{nilp}

Suppose that $G$ is a finite group, $p$ is a prime, and $B$ is a $p$-block
of $G$. In general, we use the notation in \cite{N}.
Hence $\irr B$ are the irreducible complex characters in $B$, $\ibr B$
are the irreducible Brauer characters in $B$, and ${\rm Irr}_0(B)$ are the
height zero characters of $B$.

For the sake of brevity, 
let us say that $B$ is {\bf EHZD} (equal height zero degrees)
if there is an integer $d$ such that
$\chi(1)=d$ for all $\chi \in {\rm Irr}_0(B)$.

Recall that a block $B$ is {\bf nilpotent} if whenever $(Q,b_Q)$ is a
$B$-subpair (that is, $b_Q$ is a block of $Q\cent GQ$  such that $(b_Q)^G=B$),
then $\norm G{Q,b_Q}/\cent GQ$ is a $p$-group.

If $B$ is nilpotent, then we know that ${\rm IBr}(B)=\{ \phi\}$ by Theorem~(1.2)
of \cite{B-P}. Also, if  $\chi \in \irr B$ has height zero, then by
(3.11) in page 126 of \cite{B-P}, we have that $\chi(1)=\phi(1)$.
It then follows that all irreducible height zero
characters in $B$ have the same degree. Thus, as we mentioned in the introduction,
 nilpotent blocks are EHZD blocks.
 (We also notice here that in a nilpotent block
 all height zero characters are modularly irreducible.
 This condition, if not equivalent, seems also closely related 
 to nilpotency as we shall point out in several places
 of this paper.)

\section{Principal blocks} \label{ppal}

If $B$ is the principal block of $G$,
then $(Q,b_Q)$ is $B$-subpair if and only if
$b_Q$ is the principal block of $\norm GQ$ (by the Third Main Theorem).
Since the principal block $b_Q$ is $\norm GQ$-invariant,
we conclude that $B$ is nilpotent if and only if $\norm GQ/\cent GQ$
is a $p$-group for every $p$-subgroup $Q$ of $G$. Hence
 $B$ is nilpotent
if and only if $G$ has a normal $p$-complement, by a classical theorem of
Frobenius.

\medskip

\begin{thm}
 Let $G$ be a finite group, let $p$ be a prime and let $B$ be the principal
 block of $G$. Then the following conditions are equivalent:
 \begin{enumerate}
  \item[\rm(a)] All height zero $\chi \in \irr B$ have the same degree.
  \item[\rm(b)] All height zero $\chi \in \irr B$ are modularly irreducible.
  \item[\rm(c)] $B$ is a nilpotent block.
 \end{enumerate}
\end{thm}

\begin{proof}
In Section 2 we have pointed out that (c) implies
(a) and (b).  Suppose now that all height zero characters in
$B$ have the same degree. Hence all non-linear characters in $B$ have
degree divisible by $p$. Then $G$ has
a normal $p$-complement by Corollary 3  of \cite{I-S} and so $B$ is nilpotent.

Now, suppose that all  the height zero (that is, $p'$-degree) characters
in $B$ lift an irreducible Brauer character of $G$.
We are going to use a theorem of Pahlings that asserts
that $\phi \in \ibr G$ is linear and  all nonlinear
characters
$\chi \in \irr G$
with decomposition number $0 \ne d_{\chi \phi}$ have degrees divisible by
$p$, 
then $G$ has a normal $p$-complement.
(See Theorem 2 of \cite{P}.)
Write  $\phi=1_G \in \ibr G$ for the trivial Brauer character
of $G$, and suppose that $\chi$ is
non-linear with $d_{\chi \phi} \ne 0$.
If $\chi$ has $p'$-degree, then by hypothesis,
$\chi^0=1_G$
and therefore $\chi$ is linear. This is not possible.
Hence, we conclude that $p$ divides $\chi(1)$.
It follows that $G$ has a normal $p$-complement
by Pahling's theorem.
\end{proof} 
\section{Abelian Defect Groups} \label{abelian}

In this section we prove that
EHZD blocks with abelian defect groups are
exactly the nilpotent blocks (assuming 
Brauer's Height Zero Conjecture).

\begin{thm}
 Let $B$ be a block with an abelian defect group,
 and assume that $\irr B={\rm Irr}_0(B)$. 
 Then the following conditions are equivalent:
 \begin{enumerate}
  \item[\rm(a)] All height zero $\chi \in \irr B$ have the same degree.
  \item[\rm(b)] All height zero $\chi \in \irr B$ are modularly irreducible.
  \item[\rm(c)] $B$ is a nilpotent block.
 \end{enumerate}

\end{thm}
 
\begin{proof}
By Proposition~1 of \cite{O-T}, we have that (a) and (b) are equivalent.
Also, by Theorem 3 of \cite{O-T}, we have that (a)
happens if and only if $B$ has inertial index one.  
Hence, it suffices to show that the nilpotent blocks
with abelian defect groups are exactly the
blocks with inertial index one and abelian defect groups.
In page 118 (1.ex.3) of \cite{B-P} it is stated that
blocks with abelian defect group and inertial index one
are nilpotent. Now, suppose that  $B$ is nilpotent with 
defect group $D$. If $(D, b_D)$ is a $B$-subpair, then $\norm G{D,b_D}/\cent GD$
is a $p'$-group  (by Theorem (9.22) of \cite{N}).
Since $B$ is nilpotent, $\norm G{D,b_D}/\cent GD$ is also a $p$-group,
and we conclude that $\norm G{D,b_D}=\cent GD$. That is, $B$ has inertial index one.
\end{proof}

\section{Normal Defect Groups} \label{normal}

In this Section we prove that
EHZD  blocks with a normal defect group are nilpotent. 

\medskip
The following
should be well-known.

\medskip

\begin{lem}\label{lndf}
 Let $B$ be a block with defect group $D\nor G$ and let $b_D$ be a block of
 $D\cent GD$ covered by $B$. Let $\tilde B$ be the Fong-Reynolds correspondent
 of $B$ over $b_D$. If $\tilde B$ is nilpotent, then $B$ is nilpotent.
\end{lem}

\begin{proof}
Let $T$ be the stabilizer of $b_D$ in $G$. Now $(b_D)^T=\tilde B$,
and $(D,b_D)$ is a $\tilde B$-subpair. Since
$\tilde B$ is nilpotent, we have that  $T/\cent GD$ is a $p$-group. Since $T/D\cent GD$
has order not divisible by $p$, we conclude that $T=D\cent GD$ and $\tilde B=b_D$.
Now, suppose that $(Q,b_Q)$ is a $B$-subpair. We want to show
that $\norm G{Q,b_Q}/\cent GQ$ is a $p$-group. For this we may replace
$(Q, b_Q)$ by any $G$-conjugate.
Since $(b_Q)^G=B$, we have that
$Q \sbs D$ (Theorem (4.14) of \cite{N}). Now let $e=(b_Q)^{D\cent GQ}$.
We have that $e^G=B$ by the transitivity of induction. 
Now, if $f$ is a block of $D\cent GD$ covered by $e$, 
we have that $e=f^{D\cent GQ}$ by Corollary (9.21) of \cite{N}.
Hence $e^G=B$ and $f^x=b_D$ for some $x \in G$.
Now, replacing $(Q,b_Q)$ by $(Q^x, (b_Q)^x)$, we may
assume that $(b_D)^{D\cent GQ}=e$.  Now,
suppose that $y$ stabilizes $(Q,b_Q)$. Then $y$ stabilizes $ (b_Q)^{D\cent GQ}=e$.
Hence $e$ covers $(b_D)^y$, and therefore $(b_D)^{yz}=b_D$ for
some $z \in \cent GQ$. Since $T=D\cent GD$, we see that $yz \in D\cent GD$
and therefore $y \in D\cent GQ$. Thus $\norm G{Q,b_Q}/\cent GQ$
is a $p$-group and $B$ is nilpotent.
\end{proof}

\begin{thm}   \label{thm:abdef}
 Suppose that $B$ has defect group $D \nor G$.
 Then the following conditions are equivalent:
 \begin{enumerate}
  \item[\rm(a)] All height zero $\chi \in \irr B$ have the same degree.
  \item[\rm(b)] All height zero $\chi \in \irr B$ are modularly irreducible.
  \item[\rm(c)] $B$ is a nilpotent block.
 \end{enumerate}
 \end{thm}
\medskip

\begin{proof}
We already know that (c) implies (a) and (b).
We prove by  induction on $|G|$ that (a) implies (c). In an analog way,
we could prove that (b) implies (c).
Let $b_D$ be a block of $D\cent GD$ inducing
$B$ with defect group $D$. Let $T$ be the stabilizer in $G$ of
$b_D$, and let $\tilde B$ be the Fong-Reynolds
correspondent of $B$ over $b_D$. If all height zero $\chi \in \irr B$ have the
same degree, then the same happens in $\tilde B$ by the Fong-Reynolds
correspondence \cite[Theorem (9.14)]{N}.
 By Lemma (\ref{lndf}), we may assume that $T=G$. 
Now by Reynolds Theorems 6 and 7 of \cite{R},
there exists a group $M$ with normal Sylow $p$-subgroup
$D$ such that $D\cent MD= D \times Z$, where $Z$
is central, and $M/D\cent MD \cong G/D\cent GD$.
Also  $M$ has a block $B_1$, with defect group $D$, having all height zero irreducible characters
of the same degree
and such that $\irr{B_1}= \irr{M|\lambda}$ for some
irreducible $\lambda \in \irr Z$. (We use $\irr{M|\lambda}$ to
denote the irreducible characters of $M$ lying over $\lambda$.)
Since $D$ is normal in $M$, the height zero characters of $B_1$
are exactly those irreducible characters
of $B_1$ of $p'$-degree. Hence,
these are exactly the irreducible
characters of $B_1$ having $D'$ in their kernel. Now, by Theorem (9.9.b) of \cite{N},
we have that $B_1$ contains a block $\bar B_1$
of $M/D'$ with defect group $D/D'$. This block $\bar B_1$
has all irreducible characters of the same degree. By
Theorem 3 of \cite{O-T}, we conclude
that $\bar B_1$ has inertial index one. Thus
 $$(D/D' )\cent {M/D'}{D/D'}=M/D' \, .$$ 
 Using that a $p'$-group acts trivially on $D$ if
 and only if it acts trivially on $D/D'$,
 we easily  conclude
that $M=D\cent MD$ and therefore $G=D\cent GD$.
In this case, the block is nilpotent. (See (1.ex 1) on page 118 of \cite{B-P}.)
\end{proof}

\section{Quasi-Simple Groups} \label{quasisimple}

From now until the last section of the paper,
we are  devoted to proving the following result.

\begin{thm}   \label{main}
 Let $S$ be a finite non-abelian simple group, $G$ a quasi-simple group with
 $G/Z(G)\cong S$, and $p$ a prime. Assume that $B$ is a $p$-block of $G$
 such that all characters in $\Irr_0(B)$ have the same degree. Then the defect
 group of $B$ is abelian and thus $B$ is nilpotent, unless possibly one of
 the following holds:
 \begin{enumerate}
  \item $B$ is a faithful block for the 2-fold covering group $2.\fA_n$ of
   the alternating group $\fA_n$ ($n\ge14$) (a so-called spin-block), or
  \item $B$ is a quasi-isolated block for an exceptional group of Lie
   type and $p$ is a bad prime.
 \end{enumerate}
\end{thm}

Theorem~\ref{thm:abdef} shows that in order to check Theorem~\ref{main} it
suffices to prove the following for any block $B$ of $G$ all of whose
characters in $\Irr_0(B)$ have the same degree:
\begin{enumerate}
 \item the defect group of $B$ is abelian, and
 \item $B$ satisfies the Height Zero Conjecture.
 \end{enumerate}

In order to do this, we invoke the classification of finite
simple groups as well
as the Deligne--Lusztig theory of characters of finite reductive groups and
the fundamental results of Bonnaf\'e--Rouquier and Cabanes-Enguehard on blocks.
It will be given in several steps in the subsequent sections. The proof in
the case of alternating groups will lead to a relative hook formula for the
character degrees in $p$-blocks of the symmetric group.

\section{Unipotent blocks} \label{sec:unip}

In this section we consider the unipotent blocks of finite groups of Lie type.
We introduce the following standard setup: Any non-exceptional
Schur covering group of a finite simple group of Lie type can be
obtained as $G:=\bG^F$, where $\bG$ is a simple algebraic group of
simply-connected type over the algebraic closure of a finite field, and
$F:\bG\rightarrow\bG$ a Frobenius map with finite group of fixed points
$\bG^F$, with the sole exception of the Tits group $\tw2F_4(2)'$, which will
be treated later in Proposition~\ref{prop:spor}. Let $\bG^*$
denote a group in duality with $\bG$ and with corresponding Frobenius map
$F^*:\bG^*\rightarrow\bG^*$ and fixed points $G^*:={\bG^*}^{F^*}$. Let
$r$ denote the defining characteristic of $G$ and $q$ the absolute value of
all eigenvalues of Frobenius on the character lattice of an $F$-stable torus
of $G$, a half-integral power of $r$. We will then also write $G=G(q)$ in order
to indicate the corresponding value of $q$. \par
By the fundamental work of Lusztig, the irreducible characters of $G$ can be
partitioned into so called Lusztig series
$$\Irr(G)=\coprod_{(s)}\cE(G,s)$$
indexed by conjugacy classes of semisimple elements $s$ in $G^*$. The
characters in the Lusztig series $\cE(G):=\cE(G,1)$ corresponding to the
trivial element in $G^*$ are the so-called unipotent characters. These can be
viewed as being the building blocks of the ordinary character theory of finite
groups of Lie type. Again by results of Lusztig, the unipotent characters can
be parametrized by a set depending only on the type of $G$, that is, on the
Weyl group of $\bG$ together with the action of $F$ on it, not on $q$ or~$r$.
Moreover, their degrees are given by the value at $q$ of polynomials in one
indeterminate of the form
$$\frac{1}{n}x^{a}\prod_{i=1}^m\Phi_i(x)^{a_i},$$
where $n$ is either a power of~$2$ or a divisor of~$120$, and $\Phi_i(x)$
denotes the $i$th cyclotomic polynomial over $\QQ$ (see for example Chapter~13
of \cite{Ca}). We write $\Deg(\gamma)\in\QQ[x]$ for the degree polynomial of a
unipotent character $\gamma\in\cE(G)$.  \par

\subsection{Specializations of degree polynomials} \label{subsec:exc}
We start by investigating specializations of degree polynomials of unipotent
characters.
We first discuss the question when two different degree polynomials $f_1,f_2$
can lead to the same character degree $f_1(q)=f_2(q)$.

\begin{lem}   \label{lem:speceq}
 Let $n_1,n_2,m\in\NN$, $a_i\in\ZZ$ with $a_m\ne0$, $q>1$ a prime power, and
 assume that
 $$n_1=n_2\prod_{i=1}^m \Phi_i(q)^{a_i}.$$
 \begin{itemize}
  \item[\rm(a)] If both $n_1,n_2$ are powers of~2, then $(m,q)=(6,2)$ and
   $(a_3,\ldots,a_6)=(0,0,0,-a_2)$, or $(m,q)=(2,3)$, or $m=1$, $q=2^f+1$.
  \item[\rm(b)] If both $n_1,n_2$ are divisors of~120,
   then either $m=4$ and $q\in\{2,3\}$, or $m=2$ and
   $q\in\{2,3,4,5,7,9,11\}$, or $m=1$, $q-1|120$.
 \end{itemize}
\end{lem}

\begin{proof}
First assume that there is a Zsigmondy primitive prime divisor $p_m$ of
$\Phi_m(q)$, that is, $p_m$ divides $\Phi_m(q)$, but it does not divide
$\Phi_i(q)$ for $i<m$. This is the case unless $m=1$, or $m=2$ and $q+1$ is a
power of~2, or $(m,q)=(6,2)$. Clearly, $p_m\ne2,3$ if $m\ne1,2$, and $p_m\ne5$
if $m\ne1,2,4$.
Comparing prime factorizations on both sides we see that $m\in\{1,2,4\}$ if
$n_1n_2$ is divisible by~5, and $m\in\{1,2\}$ if not. \par
Further assume that $m=4$, so $5|(q^2+1)$. As any two of $q-1,q+1,q^2+1$ have
$\gcd$ at most~2, their only prime divisors can be 2,3 and~5, and we must have
$q^2+1\le240$. It is easy to check that this only happens for $q\in\{2,3\}$.
Next assume that $m=2$, $q+1$ is not a power of~2, and~5 is the largest
Zsigmondy prime for $q+1$. As before $q+1\le240$, and $q-1,q+1$ are only
divisible by 2,3 and~5. We arrive at $q\in\{4,9\}$. Similarly, if
$m=2$ and~3 is the only Zsigmondy prime for $q+1$, then $q\in\{2,5,11\}$.
\par
Thus we may assume that there is no Zsigmondy prime for $\Phi_m(q)$, that is,
$m=1$, or $m=2$ and $q+1$ is a power of~2, or $(m,q)=(6,2)$. In the latter
case, using Zsigmondy primes for $\Phi_3(q),\Phi_4(q),\Phi_5(q)$ we see that
$a_3=a_4=a_5=0$. Moreover, as $\Phi_2(2)=3=\Phi_6(2)$, these two factors must
occur with opposite exponent. If $m=2$, then both $q-1,q+1$ have to be powers
of~2, whence $q=3$, or both have to divide~240. The latter implies that
$q\in\{3,7\}$. Finally, when $m=1$ then $q-1$ is a power of~2 or a divisor
of 120. This proves the claim.
\end{proof}

\begin{prop}   \label{prop:degexc}
 Let $f_1,f_2\in\QQ[x]$ be the degree polynomials of two unipotent characters
 of an exceptional group of Lie type $G=G(q)$. If $f_1(q)=f_2(q)$ for some
 prime power $q>1$, respectively square root of some odd power of a prime for
 the Suzuki or Ree groups, then $f_1=f_2$, or
 $$G\in\{G_2(2),\ \tw2B_2(2),\ \tw2F_4(2),\ \tw2G_2(3)\}.$$
\end{prop}

\begin{proof}
Write $f_j=\frac{1}{n_j}x^{a_j}\prod_{i=1}^{m_j}\Phi_i(x)^{a_{i,j}}$ for
$j=1,2$. According to \cite[Chap.~13]{Ca}, $n_j|120$ for $G=E_8$,
and $n_j|24$ else. Now $f_1(q)=f_2(q)$ implies that
$q^{a_1-a_2}\prod_{i=1}^{m_j}\Phi_i(q)^{a_{i,1}-a_{i,2}}\in\QQ$ has numerator
and denominator a divisor of~120. Since the second factor is coprime to $q$,
this holds in fact for both factors. Then Lemma~\ref{lem:speceq}(b) shows that
$q\le 121$ for $G=E_8$, and $q\le25$ for the other types. For these finitely
many values of $q$ and finitely many types the assertion can be checked
from the tables of degree polynomials. In fact, the additional restrictions
in Lemma~\ref{lem:speceq}(b) allow to restrict the number of necessary
computations even further.
\end{proof}

Note that none of the exceptions $G_2(2),\tw2B_2(2),\tw2G_2(3),\tw2F_4(2)$ is
a perfect group.
Unfortunately, there are infinitely many exceptions to the conclusion of the
previous proposition in the case of classical groups, so we will choose a
different approach for those.

\subsection{$e$-symbols and degrees} \label{subsec:symbs}
We need to give a brief recall of the notion of $e$-symbols and associated
degree, see \cite{MaU}. \par
Let $e\ge1$ be an integer. An {\em $e$-symbol} is a sequence
$S=(S_1,\ldots,S_e)$ of $e$ strictly increasing sequences
$S_i=(s_{i1}<\ldots<s_{im})$ of non-negative integers of equal
length~$m$. The {\em rank} of an $e$-symbol $S$ is defined as
$$\rk(S):=\sum_{s\in S}s -e\binom{m}{2}.$$
We define an equivalence relation on $e$-symbols as the reflexive, symmetric
and transitive closure of the relation $\sim$ given by
$$(S_1,\ldots,S_e)\sim (S_1',\ldots,S_e')\ \Longleftrightarrow\ 
  S_i'=(0,s_{i1}+1,\ldots,s_{im}+1).$$
There is a natural 1-1 correspondence between $e$-tuples of partitions
$\pi=(\pi_1,\ldots,\pi_e)\vdash r$ of $r$ and equivalence classes
of $e$-symbols of rank~$r$, as follows: by adding zeros we may assume that all
$\pi_i=(\pi_{i1}\le\ldots\le\pi_{im})$ have the same number of
parts. It is easily verified that $S(\pi)=(S_1,\ldots,S_e)$, with
$S_i:=(\pi_{i1},\pi_{i2}+1,\ldots,\pi_{im}+m-1)$ for $1\le i\le e$,
has indeed rank~$r$, and is well-defined up to equivalence. 
\par
Let $(v;u_1,\ldots,u_e)$ be indeterminates over $\QQ$. For an $e$-symbol $S$
we define
$$f_S:=(-1)^{c(S)}\frac{\displaystyle{(v-1)^r\prod_{i=1}^e u_i^r\cdot
   \prod_{i=1}^e\,\prod_{j=i}^e\,\prod_{s\in S_i}\,
   \prod_{\stackrel{t\in S_j}{s>t\text{ if }i=j}}(v^su_i-v^tu_j)}}
  {\displaystyle{v^{a(S)}\prod_{i<j}(u_i-u_j)^m\cdot
   \prod_{i,j=1}^e\,\prod_{s\in S_i}\,\prod_{k=1}^s(v^ku_i-u_j)}},$$
where
$$c(S):=\binom{e}{2}\binom{m}{2}+r(e-1),\quad\text{and }\ 
  a(S):=\sum_{i=1}^{m-1}\binom{ei}{2}$$
(see \cite[(5.12)]{MaU}). It can be checked that the rational function $f_S$
only depends on the equivalence class of the $e$-symbol $S$. We shall also
write $f_\pi$ for $f_S$ with $S=S(\pi)$.
\par
The following connection to the imprimitive complex reflection group
$G(e,1,r)\cong C_e\wr\fS_r$ will be important for us. The irreducible complex
characters of the wreath product $G(e,1,r)$ can be parametrized by $e$-tuples
of partitions $(\pi_1,\ldots,\pi_e)$ of $r$ (see for example \cite[(2A)]{MaU}),
hence by equivalence classes of $e$-symbols of rank~$r$. Now let
$\cH=\cH(W,\bu)$ denote the cyclotomic Hecke algebra for $W=G(e,1,r)$ with
parameters $\bu=(v;u_1,\ldots,u_e)$. This carries a canonical symmetrizing
form. By the main result of Geck--Iancu--Malle \cite{GIM} the Schur element
(with respect to this form) of
the irreducible character of $\cH$ indexed by the multipartition
$(\pi_1,\ldots,\pi_e)\vdash r$ is $f_S^{-1}$, where $S=S(\pi_1,\ldots,\pi_e)$
(see Conjecture~2.20 in \cite{MaU}).
\par
In particular, specializing $v$ to~1 and $u_j$ to the $e$th roots of unity
$\zeta_j:=\exp(2\pi ij/e)$ we obtain
$$f_{S}(1;\zeta_1,\ldots,\zeta_e)=\frac{d_S}{|G(e,1,r)|}
  =\frac{d_S}{e^r\, r!},$$
where $d_S$ denotes the degree of the irreducible character of $G(e,1,r)$
indexed by $S$.
\par
For later use let's record the following special cases. If $r=1$,
so $G(e,1,r)$ is the cyclic group $C_e$, then a multipartition
$\pi=(\pi_1,\ldots,\pi_e)\vdash r$ is uniquely determined
by the unique $i$ such that $\pi_i=(1)$. The corresponding $e$-symbol $S$ has
$S_i=(1)$, $S_j=(0)$ for $j\ne i$, and
$$f_{S}=\prod_{j\ne i}\frac{ u_j}{u_j-u_i}$$
(compare \cite[Bem.~2.4]{BM93}).  \par
More generally the $e$-symbols with $S_i=(r)$, $S_j=(0)$ for $j\ne i$
parametrize linear characters $\phi_i$ of $G(e,1,r)$, for $1\le i\le e$.
Evaluation of the defining formula shows that then
\begin{equation} \label{eq:linchar}
  f_{S}=\prod_{k=1}^r\left(\frac{v-1}{v^k-1}\cdot\prod_{j\ne i}
        \frac{u_j}{u_j-v^{k-1}u_i}\right).
\end{equation}

\subsection{$d$-Harish-Chandra series and cyclotomic Hecke algebras}
\label{subsec:class}
The blocks of finite groups of Lie type are closely related to so-called
$d$-Harish-Chandra series. Let $G$ be as above, the group of fixed points of
a simple algebraic group $\bG$ under a Frobenius map. For any $d\in\NN$, there
is a notion of $d$-split Levi subgroup $\bL$ of $\bG$ (an $F$-stable Levi
subgroup of $\bG$), and of $d$-cuspidal unipotent character of $L:=\bL^F$, see
for example \cite{BM98}. A pair $(L,\lambda)$ consisting of a $d$-split Levi
subgroup $L\le G$ with a $d$-cuspidal unipotent character
$\lambda\in\cE(L)$ of $L$ is called a $d$-cuspidal pair. Its {\em relative
Weyl group} is then defined as
$$W_G(L,\lambda):=\No_G(\bL,\lambda)/L.$$
By Brou\'e--Malle--Michel \cite[Thm.~3.2]{BMM}, the set of unipotent
characters of $G$ admits a natural partition
\begin{equation} \label{eq:HCpart}
  \cE(G)=\coprod_{(L,\lambda)/\sim}\cE(G,(L,\lambda)),
\end{equation}
into $d$-Harish-Chandra series $\cE(G,(L,\lambda))$, where $(L,\lambda)$ runs
over the $d$-cuspidal pairs in $G$ modulo conjugation. Furthermore, for each
$d$-cuspidal pair $(L,\lambda)$, there is a bijection
\begin{equation} \label{eq:HCmap}
  \rho(L,\lambda):\cE(G,(L,\lambda))\oneone  \Irr(W_G(L,\lambda))
\end{equation}
between its $d$-Harish-Chandra series and the irreducible characters of its
relative Weyl group $W_G(L,\lambda)$. The degree polynomials are then given
by the following $d$-analogue of Howlett-Lehrer-Lusztig theory:

\begin{thm}  \label{thm:HLL}
 Let $(L,\lambda)$ be a $d$-cuspidal pair of $G$. Then for any
 $\phi\in\Irr(W_G(L,\lambda))$ there exists a rational function
 $D_\phi(x)\in\QQ(x)$ with zeros and poles only at roots of unity or
 zero, but not at primitive $d$th roots of unity, satisfying
 $$\Deg(\gamma)=\pm|G:L|_{x'} D_{\rho(L,\lambda)(\gamma)}\,\Deg(\lambda)\qquad
   \text{ for all }\gamma\in\cE(G,(L,\lambda)),$$
 and
 $$|G:L|_{x'}D_\phi\equiv|W_G(L):W_G(L,\lambda)|\phi(1)\pmod{\Phi_d(x)}\quad
   \text{ for all }\phi\in\Irr(W_G(L,\lambda)).$$
\end{thm}

See \cite[Thm.~4.2]{MaH} and the references given there. In fact, the
$D_\phi(x)$ are inverses of Schur elements of a cyclotomic Hecke algebra
attached to $W_G(L,\lambda)$ with respect to its canonical symmetrizing form.
For example, if $W_G(L,\lambda)\cong G(e,1,r)$, then $D_\phi(x)$ is a suitable
specialization of $f_\phi(\bu)$ as defined above.

We first determine for which parameters $(u_1,\ldots,u_e)$ all Schur elements
of the cyclotomic Hecke algebra for the cyclic group $G(e,1,1)$ are equal:

\begin{lem}   \label{lem:cyclic}
 Let $K$ be a field of characteristic~0 and $u_i\in K^\times$, $1\le i\le e$,
 pairwise distinct.
 If $\prod_{\stackrel{k=1}{k\ne i}}^e u_k/(u_k-u_i)$ is independent of $i$,
 then there exists $y\in K^\times$ with
 $$\{u_i\mid 1\le i\le e\}=\{\zeta^iy\mid 1\le i\le e\},$$
 where $\zeta\in K$ is a primitive $d$th root of unity.
\end{lem}

\begin{proof}
Equivalently we may assume that $u_i\prod_{k\ne i}(u_i-u_k)$
is independent of $i$. Thus, with $f:=\prod_{k=1}^e (x-u_k)\in
K(u_1,\ldots,u_e)[x]$ and $'$ denoting the derivative with respect to $x$,
$$u_if'(u_i)=u_i\prod_{\stackrel{k=1}{k\ne i}}^e(u_i-u_k)=:c$$
is independent of $i$. That is, $u_1,\ldots,u_e$ are zeros
of the polynomial $g:=xf'-c$ of degree~$e$, so
$g=b\prod_{k=1}^e(x-u_k)=b f$ for some $b\in K$. Writing
$f=\sum_{j=0}^e a_j x^j$ we have $j a_j=b a_j$ for $j=1,\ldots,e$.
Since $a_e=1$ we conclude $b=e$, and thus $a_j=0$ for $j=1,\ldots,e-1$.
The claim follows.
\end{proof}

The following result will allow to show the existence of different height zero
degrees in blocks of classical groups, that is, groups of type $A_n$, $B_n$,
$C_n$, $D_n$, $\tw2A_n$ or $\tw2D_n$.
 
\begin{prop}   \label{prop:degclass}
 Let $G=G(q)$ be quasi-simple of classical type. Let $(L,\lambda)$ be a
 $d$-cuspidal pair in $G$. Assume that $W_G(L,\lambda)\ne1$. Then there exist
 unipotent characters $\gamma_1,\gamma_2\in\cE(G,(L,\lambda))$ with the
 following properties:
 \begin{itemize}
  \item[\rm(a)] $\gamma_1(1)\ne\gamma_2(1)$, and
  \item[\rm(b)] $\rho(L,\lambda)(\gamma_i)\in\Irr(W_G(L,\lambda))$ are linear
   characters.
 \end{itemize}
 More precisely, $\gamma_1(1)_q<\gamma_2(1)_q$, or $q=2$, $p=3$ and $G$ is of
 type $D_n$ or $\tw2D_n$.
\end{prop}

\begin{proof}
We will show that $W_G(L,\lambda)$ has linear characters $\phi_1,\phi_2$ such
that $D_{\phi_1}(q)\ne D_{\phi_2}(q)$. The claim then follows from
Theorem~\ref{thm:HLL}. In groups of classical type, there are three essentially
different possibilities for the
structure of the relative Weyl group (see \cite[(3B)]{BM93}). Firstly,
$W_G(L,\lambda)$ could be a symmetric group $\fS_n$.
This happens if and only if either $G=\SL_n(q)$ and $d=1$, or $G=\SU_n(q)$
and $d=2$. In both cases, all of $\cE(G)$ is just one $d$-Harish-Chandra series
and we may take the trivial and the Steinberg character, which correspond to
the two linear characters of $\fS_n$ and have distinct degrees. \par
The second possibility is that $W_G(L,\lambda)\cong G(m,1,r)$ for some $m\ge2$.
This occurs for all
other $d$-Harish-Chandra series $\cE(G,(L,\lambda))$ in classical groups for
which $\lambda$ is not parametrized by a so-called degenerate symbol. Let
$\phi_i$ denote the linear character of $G(m,1,r)$ parametrized by the
multipartition $(\pi_1,\ldots,\pi_m)$ with $\pi_i=(r)$. According
to~(\ref{eq:linchar}) we have $D_{\phi_i}=c\ f_i^{-1}(\bq)$, where
$$f_i(v,u_1,\ldots,u_m):=u_i^r\prod_{\stackrel{j=1}{j\ne i}}^{m}
  \prod_{k=0}^{r-1}(v^ku_i-u_j)$$
for some non-zero $c$ not depending on $i$, and the parameters $\bq$ are
certain powers of $q$, up to sign, as follows
(see \cite[Bem.~2.10, 2.14, 2.19]{BM93}):
\begin{itemize}
 \item[(I)] for $G=\SL_n(q)$, $d\ne1$, we have $m=d$ and 
  $$\bq=(q^d; 1,q^{b_1d+1},q^{b_2d+2},\ldots,q^{b_{d-1}d+d-1}),$$
 \item[(I')] for $G=\SU_n(q)$, $d\ne2$, we have $m=d^*$, and $\bq$ is obtained
  from the parameters in case~(I) for $d^*$ by replacing $q$ by $-q$,
 \item[(II)] for $G$ of type $B_n,C_n,D_n,\tw2D_n$ and $d$ odd we have
  $m=2d$, $e=d$ and
  $$\bq= (q^e; 1,q^{b_1e+1},\ldots,q^{b_{e-1}e+e-1},-q^{b_ee},\ldots,
   -q^{b_{2e-1}e+e-1}),$$
 \item[(II')] for $G$ of type $B_n,C_n,D_n,\tw2D_n$ and $d\equiv2\pmod4$ we
  have $m=d$, and $\bq$ is obtained from the parameters in case~(II)
  for $d^*$ by replacing $q$ by $-q$,
 \item[(III)] for $G$ of type $B_n,C_n,D_n,\tw2D_n$ and $d\equiv0\pmod4$ we
  have $m=d$, $e=d/2$ and
  $$\bq=(-q^e; 1,q^{b_1e+1},\ldots,q^{b_{e-1}e+e-1},-(-1)^{b_e}q^{b_ee},\ldots,
   -(-1)^{b_{2e-1}}q^{b_{2e-1}e+e-1}),$$
\end{itemize}
where the $b_i$ are non-negative integers which are determined by $\lambda$.
Here, for $d\in\NN$, $d^*$ is defined by
$$d^*:=\begin{cases} 2d& \text{if $d$ is odd,}\cr
   d/2& \text{if }d\equiv2\pmod4,\cr d& \text{if }d\equiv0\pmod4.\end{cases}$$
Note that it can never happen that
$v^ku_i-u_j=0$ for the above choices of parameters. Furthermore, we claim that
there is at least one $i$ with $|u_i|>1$. Indeed, otherwise we are necessarily
in cases~(II) or (II') and $e=1$. But then, since $\lambda$ is not
parametrized by a degenerate symbol, $b_1>0$ by the definition of the $b_i$ in
\cite[(2B)]{BM93}, so $|u_2|>1$, a contradiction. \par
By our above reductions it suffices now to show that not all $f_i(\bq)$ are
equal. For this, we estimate the $q$-power in $f_i(\bq)$ for two choices
of~$i$. For $i=1$
we have $u_1=1$, and $v^ku_1-u_j$ is at most divisible by the $q$-power $u_j$
(at least when $q$ is odd), and not divisible by $q$ if $k=0$, so $f_1(\bq)$
is at most divisible by the $q$-power $\prod_{j=2}^m u_j^{r-1}$. Now let $i$
be such that $|u_i|$ is maximal among the $\{u_1,\ldots,u_m\}$. Then
$v^ku_i-u_j$ is divisible by at least the $q$-power $u_j$, so $f_i(\bq)$ is
at least divisible by
$$u_i^r\prod_{\stackrel{j=1}{j\ne i}}^m u_j^r=\prod_{j=1}^m u_j^r.$$
By our above observation we have $|u_i|\ge q$, so $f_1(\bq)\ne f_i(\bq)$ as
claimed.
\par
If $q$ is even, then $v^ku_1-u_j$ is divisible by $2u_j$ if $v^ku_1=-u_j$. For
fixed $j$, this can only happen for at most one value of $k$, and only when
$j\equiv 1+e\pmod{2e}$ and we are in cases (II), (II') or~(III). Thus, we get
an additional factor at most~2 in the $q$-part of $f_1(\bq)$. It follows that
the $q$-parts of $f_1(\bq)$ and $f_i(\bq)$ can only agree if $|u_i|=q=2$ and all
other $u_j$ have absolute value~1. Thus $e=1$, we are in cases~(II) or~(II')
and $\bq=(q;1,-q)$. But then
$$f_1(\bq)=\prod_{k=0}^{r-1}(q^k+q)\ne\prod_{k=0}^{r-1}(q^{k+2}+q)$$
(or the same with $q$ replaced by $-q$).
\par
Finally, we consider the case where $\lambda$ is parametrized by a degenerate
symbol, which can only happen in types $D_n$ and $\tw2D_n$. Then 
$W_G(L,\lambda)\cong G(2m,2,r)$, for some $m\ge1$. We denote by
$\psi_1,\ldots,\psi_m$ the $m$ distinct linear characters contained in the
restrictions of the linear characters $\phi_1,\ldots,\phi_{2m}$ from
$G(2m,1,r)$ to its normal subgroup $G(2m,2,r)$ of index~2. Evaluation of
\cite[(5.12)]{MaU} shows that $D_{\psi_i}=\tilde c\ g_i^{-1}(\bq)$ for some
constant $\tilde c$, where
$$g_i(v,u_1,\ldots,u_m):=u_i^r\prod_{\stackrel{j=1}{j\ne i}}^{m}
  \prod_{k=0}^{r-1}(v^ku_i-u_j)\qquad(1\le i\le m),$$
and the $\bq$ are certain powers of $q$, up to sign, as follows
(see \cite[Bem.~2.16, 2.19]{BM93}):
\begin{itemize}
 \item[(IV)] for $G$ of type $D_n,\tw2D_n$ and $d$ odd we have $m=e=d$ and
  $$\bq=(q^e; 1,q^{2b_1e+2},\ldots,q^{2b_{e-1}e+2e-2}),$$
 \item[(IV')] for $G$ of type $D_n,\tw2D_n$ and $d\equiv2\pmod4$ we have
  $m=d^*$, and $\bq$ is obtained from the parameters in case~(IV) for $d^*$ by
  replacing $q$ by $-q$,
 \item[(V)] for $G$ of type $D_n,\tw2D_n$, and $d\equiv0\pmod4$ we have
  $m=e=d/2$ and
  $$\bq=(-q^e; 1,q^{b_1d+2},\ldots,q^{b_{e-1}d+d-2}).$$
\end{itemize}
Clearly, unless $d=1$, we can argue as before to conclude that
$g_1(\bq)\ne g_i(\bq)$ for a suitable index~$i$. If $d=1$ then
$W_G(L,\lambda)\cong G(2,2,r)$ is the Weyl group of type $D_r$, and we are in
the principal 1-series of $G$. Here instead we take the trivial and the
Steinberg character, which have distinct degree.
\end{proof}

\subsection{Unipotent blocks} \label{subsec:unipblock}
After these combinatorial preparations we are ready to investigate unipotent
blocks of groups of Lie type $G=G(q)$; here a $p$-block of $G$ is called
unipotent if it contains at least one unipotent character of $G$.

\begin{thm}  \label{thm:unipblock}
 Let $\bG$ be a simple algebraic group of simply-connected type,
 $F:\bG\rightarrow\bG$ a Frobenius map with group of fixed points $G=\bG^F$.
 Let $B$
 be a unipotent $p$-block of $G$, where $p$ is not the defining characteristic
 $r$ of $\bG$. Then either $B$ is of defect~0, or $B$ contains two height~0
 characters of different degrees. Moreover, these two degrees have different
 $r$-parts, unless possibly if $r=2$.
\end{thm}

\begin{proof}
First assume that $p$ is a good prime for $G$, odd, and not equal to~3 when
$G$ is not of type $\tw3D_4$. Then by Cabanes--Enguehard \cite[Thm.~22.9]{CE}
the intersections of unipotent $p$-blocks with $\cE(G)$ are just the
$d$-Harish-Chandra series,
where $d$ is the multiplicative order of $q$ modulo~$p$. Let $B$ be a
unipotent $p$-block corresponding to the $d$-Harish-Chandra series of the
$d$-cuspidal pair $(L,\lambda)$. If $L=G$, so $\lambda$ is a $d$-cuspidal
character of $G$, then the defect group of $B$ is trivial by
\cite[Thm.~22.9(ii)]{CE}, whence $B$ is of defect~0.
\par
If $L<G$, then $W_G(L,\lambda)\ne1$. Now
$$\Deg(\gamma) =\pm|G:L|_{x'}\, D_\phi\, \Deg(\lambda)$$
for $\gamma\in\cE(G,(L,\lambda))$, where $\phi:=\rho(L,\lambda)(\gamma)$, and
$$|G:L|_{q'}D_\phi(q)\equiv|W_G(L):W_G(L,\lambda)|\phi(1)\pmod{\Phi_d(q)}$$
for $\phi\in\Irr(W_G(L,\lambda))$, by Theorem~\ref{thm:HLL}(b). As $p$ divides
$\Phi_d(q)$ by definition, this implies the same congruence $\pmod p$. By the
description in \cite[Thm.~22.9(ii)]{CE}, some unipotent character in $B$ is
of height zero. This shows that the unipotent characters in $B$ of height~0
are precisely those $\gamma\in\cE(G,(L,\lambda))$ with
$\phi=\rho(L,\lambda)(\gamma)$ of degree prime to~$p$, for example the linear
characters of $W_G(L,\lambda)$. For $G$ of classical type it is shown in
Proposition~\ref{prop:degclass} that not all unipotent characters in $B$
parametrized by linear characters of $W_G(L,\lambda)$ have the same degree.
\par
If $G$ is of exceptional type and $W_G(L,\lambda)$ is cyclic, we may invoke
Lemma~\ref{lem:cyclic} together with the parameters in \cite[Tab.~8.1]{BM93}
to conclude. The relative Weyl groups $W_G(L,\lambda)$ for exceptional groups
which are non-cyclic are listed in \cite[Tab.~3.6]{BM93} and \cite[Tab.~1]{BMM}.
It is easy to check that these have two distinct character degrees prime
to~$p$, for all primes $p$ which are good for $G$. But then the corresponding
unipotent degrees must be distinct by Proposition~\ref{prop:degexc}, and of
height~0 by Theorem~\ref{thm:HLL}(b).
\par
Next, if $G$ is of classical type and $p=2$, then all unipotent characters
of $G$ lie in the principal $p$-block of $G$, by \cite[Th.~21.14]{CE}. Here,
the trivial character and the Steinberg character have $p$-height~0 and
different degrees. \par
It remains to consider the case where $G$ is of exceptional type and $p$ is a
bad prime for $G$ (including the case of $\tw3D_4$ with $p=3$). There is no
bad prime for $\tw2B_2$. The $2$-blocks for $\tw2G_2$ and the $3$-blocks of
$\tw2F_4$ have been determined by Fong \cite{F74} resp.\ Malle
\cite[Bem.~1]{Ma90}: unipotent characters lie either in the principal block
or are of defect zero. In the principal block the trivial and the Steinberg
character have different degree. \par

\begin{table}[htbp]
\caption{Non-principal $p$-blocks of positive defect for bad $p$}
\label{tab:blockexc}
 \[\begin{array}{crllll} \hline
 G& (p,d)& L& \lambda& W_G(L,\lambda)& \gamma_1,\gamma_2\\
\hline
 F_4& (3,1)& 2^2.B_2& 1& C_2& B_2,1; B_2,\eps\\
 E_6& (3,1)& 1^2.D_4& \zeta_1& \fS_3& \\
    & (3,2)& 2.A_5& \xi& C_2& \phi_{64,4}; \phi_{64,13}\\
 E_7& (2,1)& 1.E_6& E_6[\theta]& C_2& E_6[\theta],1; E_6[\theta],\eps\\
    & (3,1)& 1^3.D_4& \zeta_1& W(B_3)& \\
 E_8& (2\text{ or }5,1)& 1^2.E_6& E_6[\theta]& W(G_2)& E_6[\theta],\phi_{1,0};
                                   E_6[\theta],\phi_{1,6}\\
    & (3\text{ or }5,1)& 1^4.D_4& \zeta_1& W(F_4)& \\
    & (3\text{ or }5,1)& 1.E_7& E_7[\xi]& C_2& E_7[\xi],1; E_7[\xi],\eps\\
    & (5,4)& 4^2.D_4& \xi_1,\ldots,\xi_4& G_8& \\
\hline
\end{array}\]\end{table}

For the other types of exceptional groups,
we use the description of unipotent blocks for bad primes $p$ obtained by
Enguehard \cite[Th.~A]{En00}. Here, again any unipotent $p$-block is either
of defect~0, or it contains at least one non-trivial $d$-Harish-Chandra series.
According to loc.\ cit.\ and the tables in \cite[pp.~347--358]{En00}, the
non-principal unipotent blocks not of defect zero are as listed in
Table~\ref{tab:blockexc} (up to Ennola duality and algebraic conjugacy; the
notation is as in loc.~cit.) In each case either the relative Weyl group has
two distinct character degrees prime to~$p$, in which case we may conclude as
above, or we list two unipotent characters $\gamma_1,\gamma_2$ in the
corresponding $d$-Harish-Chandra series which are of $p$-height~0 and have
distinct $r$-parts in their degrees (see \cite[Tab.~2]{BMM} for the list of
$d$-Harish-Chandra series and \cite[Ch.~13]{Ca} for the degrees of unipotent
characters).

This completes the proof of Theorem~\ref{thm:unipblock}. Note that the results
hold even when the finite group $G$ is not perfect, or even solvable.
\end{proof}

We now extend the result to unipotent blocks of arbitrary finite connected
reductive groups. 

\begin{thm}  \label{thm:unipblockgeneral}
 Let $\bG$ be a connected reductive group with a Frobenius map
 $F:\bG\rightarrow\bG$ and group of fixed points $G:=\bG^F$. Let $B$ be a
 unipotent $p$-block of $G$, where $p$ is not the defining characteristic $r$
 of $\bG$. Then either $B$ is of central defect, and all characters of $B$
 have the same degree, or $B$ contains two height~0
 characters of different degrees. Moreover, these two degrees have different
 $r$-parts unless possibly if $r=2$.
\end{thm}

\begin{proof}
The derived group $[\bG,\bG]$ is semisimple, hence a central product
$\bG_1\circ\ldots\circ\bG_r$ of simple algebraic groups. We assume the $\bG_i$
ordered such that $\bG_1,\ldots,\bG_s$, for some $s\le r$, is a system of
representatives for the
$F$-orbits on $\{\bG_1,\ldots,\bG_r\}$. Then, $G':=[\bG,\bG]^F$ is a central
product $G_1\circ\ldots\circ G_s$ of groups, with $G_i\cong \bG_i^{F^{m_i}}$,
where $m_i$ is the size of the $F$-orbit of $\bG_i$. Note that, in general,
$G'$ will be larger than the commutator subgroup of $G$. Modulo
$Z([\bG,\bG])^F=Z(G')$ we obtain a direct product
$$\bar G:=G'/Z([\bG,\bG])^F\cong \bar G_1\times\ldots\times \bar G_s,$$
where $\bar G_i:=G_i/(G_i\cap Z([\bG,\bG])^F)$.
Now let $B$ be a unipotent $p$-block of $G$. Since unipotent characters restrict
irreducibly to the $F$-fixed points of the derived group \cite{Lu88}
$B$ covers a unique block $B'$ of $G'$. Furthermore unipotent characters have
the center in their kernel, so the same holds for unipotent blocks. Thus
$B'$ corresponds to a unique block $\bar B$ of the direct product $\bar G$.
This is a direct product $\bar B=\bar B_1\times\ldots\times\bar B_s$ of blocks
$\bar B_i$ of $\bar G_i$.\par
Now assume that one of the $\bar B_i$ is not of defect~0 for $\bar G_i$. Since
$\bar G_i$ is a central factor group of a group as in
Theorem~\ref{thm:unipblock}, $\bar B_i$ then contains two height~0 unipotent
characters of different degrees, with different $r$-parts if $r=2$. Thus, the 
same is true for $\bar B$, hence also for $B'$. By the above-mentioned
irreducibility of restrictions, this then also holds for $B$.
\par
On the other hand, if all $\bar B_i$ are of defect~0, then so is $\bar B$, so
$B'$ is of central defect, contained in $Z([\bG,\bG])^F$. But
$Z([\bG,\bG])\subseteq Z(\bG)$ as $\bG=[\bG,\bG]Z(\bG)$, so the block $B$ is 
also of central defect in $G$, as claimed. Moreover, as each $\bar B_i$
contains a unique ordinary character, we also have $\Irr(B')=\{\chi'\}$ for
some ordinary (unipotent) character $\chi'$ of $G'$. Since this is the
restriction of an irreducible character of $G$, and $G/G'$ is abelian, all
characters of $G$ above $\chi'$ have the same degree, hence the height zero
conjecture holds for $B$ in this case.
\end{proof}

\begin{prop}   \label{prop:discon}
 Let $G$ be a finite group, $N\triangleleft G$ a normal subgroup of index
 prime to $p$ with $G/N$ either cyclic or a Klein four group. Let $b$ be a
 $p$-block of $G$, and $b'$ a $p$-block of $N$ lying below $b$. Then:
 \begin{enumerate}
  \item[\rm(a)] $b$ and $b'$ have isomorphic defect groups.
  \item[\rm(b)] Assume that $b'$ has two height~0 characters $\chi_1,\chi_2$ of
   different degrees, and let $r$ be a prime for which the $r$-parts of
   $\chi_1(1),\chi_2(1)$ differ. If $\gcd(r,|G:N|)=1$ then $b$ also has two
   height~0 characters of different degrees.
 \end{enumerate}
\end{prop}

\begin{proof}
The first assertion is well-known. For the second, let $\nu_i$ be a character
of $b$ lying above $\chi_i$, for $i=1,2$. Since $p$ does not divide the index
$|G:N|$, $\nu_i$ is again of height~0. Furthermore, the assumption that
$\gcd(r,|G:N|)=1$ and $G/N$ is cyclic or Klein four implies by Clifford theory
that $\nu_1(1)\ne\nu_2(1)$.
\end{proof}

\section{Blocks of groups of Lie type} \label{sec:Lie}

\begin{prop}   \label{prop:defchar}
 The assertion of Theorem~\ref{main} holds when $S$ is a simple group of Lie
 type and $p$ is the defining characteristic.
\end{prop}

\begin{proof}
By the result of Humphreys \cite{Hum} the covering group $G$ of $S$ has exactly
one $p$-block of defect zero, consisting of the Steinberg character, and all
other $p$-blocks are of full defect, in one-to-one correspondence with the
irreducible characters of $Z(G)$. For the principal block, it is clear that
there exist two height~0 characters of distinct degree (viz. the trivial
character and at least one further non-linear character). For the remaining
blocks, some more work is needed. By the above we may now assume that
$Z(G)\ne1$, so in particular $p$ is odd for classical groups not of types
$A_n$ or $\tw2A_n$. 
\par
Recall that $Z(G)$ is naturally isomorphic to the commutator factor group
$G^*/[G^*,G^*]$ of the dual group $G^*$ of $G$. If $s\in G^*$ is semsimple,
the corresponding semisimple character $\chi_s\in\Irr(G)$ is of $p'$-degree
given by $\chi_s(1)=|G:\Ce_{G^*}(s)|_{p'}$
(see for example \cite[(2.1)]{MaH}; note that $\Ce_{G^*}(s)$ is not necessarily
connected). So we are done if we can find two semisimple elements
$s_1,s_2\in G^*$ whose centralizer orders have different $p'$-part.

\begin{table}[htbp]
\caption{Tori and Zsigmondy primes for classical groups}
\label{tab:toriclass}
 $\begin{array}{lccll}
\hline
 & |T_1|& |T_2|& \ell_1& \ell_2\cr
\hline
 A_n& (q^{n+1}-1)/(q-1)& q^n-1& l(n+1)& l(n)\cr
 \tw2A_n\ (n\ge2$ even$)& (q^{n+1}+1)/(q+1)& q^n-1& l(2n+2)& l(n)\cr
 \tw2A_n\ (n\ge3$ odd$)& (q^{n+1}-1)/(q+1)& q^n+1& l(n+1)& l(2n)\cr
 B_n,C_n\ (n\ge2$ even$)& q^n+1& (q^{n-1}+1)(q+1)& l(2n)& l(2n-2)\cr
 B_n,C_n\ (n\ge3$ odd$)& q^n+1& q^n-1& l(2n)& l(n)\cr
 D_n\ (n\ge4$ even$)& (q^{n-1}-1)(q-1)& (q^{n-1}+1)(q+1)& l(n-1)& l(2n-2)\cr
 D_n\ (n\ge5$ odd$)& q^n-1& (q^{n-1}+1)(q+1)& l(n)& l(2n-2)\cr
 \tw2D_n (n\ge4)& q^n+1& (q^{n-1}+1)(q-1)& l(2n)& l(2n-2)\cr
\hline
\end{array}$\end{table}

In Table~\ref{tab:toriclass} we have listed two maximal tori $T_1$, $T_2$ of
$G^*$ for each type of classical group $G$ (by giving their orders, which 
determines
them uniquely). Except for types $B_n,C_n$ with $n$ even this is Table~3.5 in
Malle \cite{MaT}. We write $l(m)$ for a Zsigmondy prime divisor of $q^m-1$.
Then $|T_i|$ is divisible by the Zsigmondy prime $\ell_i$ as indicated in the
table, which exists unless $G$ is of type $A_1$, or $G$ is of type $A_2$,
$\tw2A_2$ or $B_2$ and $i=2$. (Note that the case that $q=2$ and $G$ of type
$A_5$, $A_6$ or $\tw2A_6$ does not concern us here, since then the center of
$G$ is trivial.) If $|T_i|$ is divisible by a Zsigmondy prime
$\ell_i$, then there exist regular semisimple elements $s_i$ of order $\ell_i$
in $G^*$, that is, elements with centralizer order $|\Ce_{G^*}(s_i)|=|T_i|$.
If both Zsigmondy primes exist, this yields two semisimple characters of
different degrees, and we are done. \par
So now assume that $G$ is of type $A_1$, $A_2$, $\tw2A_2$ or $B_2$. From the
known character tables it can be seen that the group
$\SL_2(q)$ has faithful irreducible characters of degrees $q+1$, $(q-1)/2$ for
$q\equiv1\pmod4$, and of degrees $q-1$, $(q+1)/2$ for $7\le q\equiv3\pmod4$.
The group $\SL_3(q)$, $q\equiv1\pmod3$, has faithful irreducible characters of
degrees $q^2+q+1$ and $(q-1)(q^2-1)$, the group $\SU_3(q)$, $2<q\equiv2\pmod3$,
has faithful irreducible characters of degrees $q^2-q+1$ and $(q+1)(q^2-1)$,
and the group $\Sp_4(q)$, $q$ odd, has faithful irreducible characters of
degrees $(q^2-1)/2$ and $q^4-1$, which are clearly distinct and of $p$-height
zero.
\par
The only exceptional simply-connected groups with non-trivial center are those
of types $E_6$, $\tw2E_6$ and $E_7$. For these, we may argue as above using the
maximal tori and Zsigmondy primes listed in Table~\ref{tab:toriexc}. The proof 
is  complete.
\end{proof}

\begin{table}[htbp]
\caption{Tori $T_1$ and $T_2$}
\label{tab:toriexc}
 \[\begin{array}{rccll}
\hline
 G& |T_1|& |T_2|& \ell_1& \ell_2\cr
\hline
     E_6(q)& \Phi_{12}\Phi_3& \Phi_9& l(12)& l(9)\cr
 \tw2E_6(q)& \Phi_{18}& \Phi_{12}\Phi_6& l(18)& l(12)\cr
     E_7(q)& \Phi_{18}\Phi_2& \Phi_{14}\Phi_2& l(18)& l(14)\cr
\end{array}\]\end{table}

We now turn to the non-defining primes for groups of Lie type.

According to the work of Brou\'e--Michel \cite{BrMi}, for any $p$-block $B$
of $G$
there exists a unique $G^*$-conjugacy class $[s]$ of semisimple $p'$-elements
of $G^*$, such that some irreducible representation of $B$ is in the rational
Lusztig series attached to $[s]$. Let's write $\cE_p(G,s)$ for the union of all
$p$-blocks of $G$ associated with the class of the $p'$-element $s\in G^*$.
The blocks in $\cE_p(G,1)$ are called {\em unipotent}. More generally, if
$\bG$ is disconnected, then a block of $\bG^F$ is called unipotent if it covers
a unipotent block of $(\bG^\circ)^F$.
We need the following crucial result
of Enguehard \cite[Th.~1.6]{En08}:

\begin{thm}[Enguehard]   \label{thm:eng}
 Assume that $p$ is good for $\bG$, and different from~3 if $F$ induces a
 triality automorphism on $\bG$. \par
 Let $s\in G^*$ be a semisimple $p'$-element, and $B$ a $p$-block in
 $\cE_p(G,s)$. Then there exists a reductive group
 $\bG(s)$ defined over $\FF_r$, with corresponding Frobenius map again denoted
 by $F$, and a unipotent $p$-block $b$ of $G(s):=\bG(s)^F$, such that the
 defect groups of $B$ and $b$ are isomorphic and there is a height-preserving
 bijection $\Irr(B)\rightarrow\Irr(b)$. Here, $\bG(s)^\circ$ is a group in
 duality
 with $C_{\bG^*}(s)^\circ$, and
 $\bG(s)/\bG(s)^\circ\cong C_{\bG^*}(s)/C_{\bG^*}(s)^\circ$.
\end{thm}

In the case of $p=2$ for classical groups, he proves \cite[Prop.~1.5]{En08}:

\begin{thm}[Enguehard]   \label{thm:eng2}
 Assume that $G$ is of classical type in odd characteristic. Let $s\in G^*$
 be a semisimple $p'$-elements. Then all 2-blocks in $\cE_2(G,s)$ have defect
 group isomorphic to a Sylow 2-subgroup of $\Ce_{G^*}(s)^\circ$. If moreover
 $G$ is of type $B_n$, $C_n$ or $D_n$, then $\cE_2(G,s)$ is a single 2-block.
\end{thm}

\begin{prop}
 The assertion of Theorem~\ref{main} holds if $G$ is quasi-simple of Lie-type.
\end{prop}

\begin{proof}
By Proposition~\ref{prop:defchar} we may assume that $p$ is not the defining
characteristic for $G$, and by Proposition~\ref{prop:spor} we have that
$S{\not\cong}\tw2F_4(2)'$. Furthermore, by the remarks at the beginning of
Section~\ref{sec:unip} we have that $G=\bG^F$ for some simple, simply connected
algebraic group $\bG$ with Frobenius map $F:\bG\rightarrow\bG$.
\par
Let $B$ be a $p$-block of $G$ and $s\in G^*$ semisimple such that
$B\subseteq\cE_p(G,s)$ (see above). First assume that $s$ is not
quasi-isolated in $G^*$, that is, $C_{G^*}(s)$ is a Levi subgroup of $G^*$. Then
by the result of Bonnaf\'e--Rouquier \cite[Th.~10.1]{CE} the block $B$ is
Morita-equivalent to a block $b\subseteq\cE_p(L,1)$ where $L$ is a Levi
subgroup of $G$ in duality with $C_{G^*}(s)$, and Jordan decomposition gives a
height preserving bijection from $B$ to $b$. We may then conclude by
Theorem~\ref{thm:unipblockgeneral}. \par
Next assume that $p$ is good for $G$, different from~3 if $G$ is of type
$\tw3D_4$. Then by Theorem~\ref{thm:eng} there is a group $\bG(s)$ in duality
with the centralizer $\bC:=C_{\bG^*}(s)$ of $s$ in $\bG^*$ and a height
preserving bijection between $B$ and a unipotent block $b$ of $G(s):=\bG(s)^F$ 
with the same defect group as $B$. By \cite[Cor.~2.9]{Bo05} the order
$a(s):=|\bC:\bC^\circ|$ of the component group of $\bC$ is prime to the
defining characteristic $r$ and divides the order of $s$. As $s$ is a
$p'$-element,
this implies that $a(s)$ is prime to $p$ as well. Moreover, by loc.~cit.\
$\bC/\bC^\circ$ is isomorphic to a subgroup of the fundamental group of $\bG$,
hence either cyclic or a Klein four group. Now let $b'$ be a $p$-block
of the normal subgroup $N:=(\bC^\circ)^F$ of $\bC^F=C_{G^*}(s)$ lying below $b$.
We showed in Theorem~\ref{thm:unipblockgeneral} that any unipotent block of the
connected group $N$ with non-abelian defect group contains two height~0
characters which are divisible by different powers of the defining prime~$r$.
Thus, Proposition~\ref{prop:discon} applies in this case and the claim follows.
\par
Now assume that $p=2$ and $G$ is of classical type $B_n$, $C_n$, $D_n$ or
$\tw2D_n$. Then $\cE_2(G,s)$ is a
single 2-block by Theorem~\ref{thm:eng2}. By Jordan decomposition the character
degrees in $\cE_2(G,s)$ are obtained from those in $\cE_2(C_{G^*}(s),1)$ by
multiplication with a common constant. If $C_{G^*}(s)^\circ$ is not a torus,
the trivial character and the Steinberg character in $\cE_2(C_{G^*}(s),1)$
have distinct degrees prime to~$p$ and the claim follows. On the other hand, if
$C_{G^*}(s)^\circ$ is a torus, that is, $s$ is a regular element in $G^*$,
then again by the Theorem~\ref{thm:eng2} of Enguehard the defect group of
$B$ is isomorphic to a Sylow 2-subgroup of $\Ce_{G^*}(s)^\circ$, hence abelian.
Moreover, by the result of Lusztig \cite{Lu88}, all characters in $B$ have
the same degree, whence $B$ satisfies the height zero conjecture.
\par
Thus we may assume that $G$ is of exceptional type, $p$ is a bad prime and
$s$ is quasi-isolated. There are no quasi-isolated elements for $\tw2B_2$. The
$p$-blocks for $\tw2G_2$, $G_2$, $\tw2F_4$ and $\tw3D_4$ have been determined
by Fong \cite{F74}, Hiss--Shamash \cite{HS90,HS92}, Malle \cite{Ma90},
Deriziotis--Michler \cite{DM87} respectively. The claim can be easily checked
from those results.

The remaining cases are the possible exceptions mentioned in the theorem.
\end{proof}

\section{Alternating and sporadic groups} \label{sec:alt}

In order to prove our main result for the alternating groups, we first
derive a similar statement for blocks of the symmetric group.

Recall that the irreducible characters of $\fS_n$ as well as the unipotent
characters of $\GL_n(q)$, where $q$ is any prime power, are parametrized by
partitions $\lambda\vdash n$. We write $\chi_\lambda$ resp.\ $\gamma_\lambda$
for the corresponding character of $\fS_n$, resp.\ of $\GL_n(q)$. The
following important connection between their degrees is well-known:
$\chi_\lambda$ is obtained by specializing $q$ to~1 in the degree polynomial
for $\gamma_\lambda$ (see for example the formula in \cite[13.8]{Ca} and
compare to the hook formula for $\chi_\lambda(1)$). This is sometimes referred
to by saying that $\fS_n$ is \lq the general linear group over the field with
one element\rq. \par
Furthermore, $\chi_\lambda$ and $\chi_\mu$ for two partitions
$\lambda,\mu\vdash n$ lie in the same $p$-block of $\fS_n$ if and only if
$\lambda$ and $\mu$ have the same $p$-core, which in turn happens if and only
if $\gamma_\lambda$ and $\gamma_\mu$ lie in the same $d$-Harish-Chandra
series of $\Irr(\GL_n(q))$, where $d=p$. Thus, the degrees of irreducible
characters of $\fS_n$ in a fixed $p$-block are specializations at $q=1$ of
degree polynomials of unipotent characters in a fixed $p$-Harish-Chandra series.
\par
Let $S=(S_1,\ldots,S_d)$ be a $d$-symbol. A {\em hook of $S$} is a pair
$h=(s,t)$ where
$$s\in S_i,\quad t\in\{0,\ldots,s\}\setminus S_j,\qquad
  \text{ with $j>i$ if $s=t$},$$
for some $1\le i,j,\le d$. We then also write $i(h):=i$,
$j(h):=j$, and $l(h):=s-t$. For 1-symbols, that is, $\beta$-sets of partitions,
this is just the usual notion of hook. We can now formulate the following
relative hook formula for characters in a fixed $p$-block of a symmetric group
which seems to be new:

\begin{thm}   \label{thm:hookformula}
 Let $p$ be a prime. Let $\pi\vdash n$ be a partition with $p$-core
 $\mu\vdash r$ and $p$-quotient $(\nu_1,\ldots,\nu_p)\vdash w$, with
 corresponding $p$-symbol $S$. Let $b_i$ denote the number of beads on the
 $i$th runner of the $p$-abacus diagram for $\mu$, and $c_i:=pb_i+i-1$. Then
 $$\chi_\pi(1)=\frac{n!}{ r!}\cdot
   \frac{1}{\displaystyle\prod_{h\text{ hook of }S}|pl(h)+c_{i(h)}-c_{j(h)}|}
   \cdot\chi_\mu(1)$$
 and
 $$\chi_\pi(1)/\chi_\mu(1)\equiv \psi_\nu(1)\pmod p$$
 where $\psi_\nu$ denotes the irreducible character of $C_p\wr\fS_w$
 parametrized by~$\nu$.
\end{thm}

\begin{proof}
Let $\gamma$ be the unipotent character of $\GL_n(q)$ parametrized by
$\pi$, for $q$ a prime power. Set $d:=p$.
Then $\gamma$ lies in the $d$-Harish-Chandra series above $(L,\lambda)$, where
$L\cong\GL_r(q)\times\GL_1(q^d)^w$, with $\lambda$ parametrized by
$\mu\vdash r$ and $n=r+dw$. Let
$S=(S_1,\ldots,S_d)$ be the $d$-symbol corresponding to $(\nu_1,\ldots,\nu_d)$.
According to Theorem~\ref{thm:HLL}, \cite[(2.19)]{MaU} we have
$\Deg(\gamma)/\Deg(\lambda)=\pm|G:L|_{q'} D_{\rho(L,\lambda)(\gamma)}=$
$$\pm\frac{\displaystyle\prod_{i=1}^n(q^i-1)}{\displaystyle(q^d-1)^w
   \prod_{i=1}^r(q^i-1)}\cdot\frac{\displaystyle(v-1)^w\prod_{i=1}^d u_i^w\cdot
   \prod_{i=1}^d\,\prod_{j=i}^d\,\prod_{s\in S_i}\,
   \prod_{\stackrel{t\in S_j}{s>t\text{ if }i=j}}(v^su_i-v^tu_j)}
  {\displaystyle v^{a(S)}\prod_{i<j}(u_i-u_j)^m\cdot
   \prod_{i,j=1}^d\,\prod_{s\in S_i}\,\prod_{k=1}^s(v^ku_i-u_j)}
$$
with
$$(v;u_1,\ldots,u_d)=(q^d; 1,q^{c_2},\ldots,q^{c_d})$$
(see (I) in Sect.~\ref{subsec:class} for the parameter values).
By our above remarks, specialization at $q=1$ gives the corresponding
character degrees for $\fS_n$. Note that numerator and denominator of the
expression for $\Deg(\gamma)/\Deg(\lambda)$ are indeed divisible by the same
power of $(q-1)$, viz.~$n+w+\binom{me}{2}$, so that the specialization makes
sense. We obtain
$$\begin{aligned}
  \frac{\chi_\pi(1)}{\chi_\mu(1)}&=
  \pm\frac{n!}{ r!}\cdot\frac{\displaystyle
   \prod_{i=1}^d\,\prod_{j=i}^d\,\prod_{s\in S_i}\,
   \prod_{\stackrel{t\in S_j}{s>t\text{ if }i=j}}(d(s-t)+c_i-c_j)}
  {\displaystyle \prod_{i<j}(c_i-c_j)^m\cdot
   \prod_{i,j=1}^d\,\prod_{s\in S_i}\,\prod_{k=1}^s(dk+c_i-c_j)}\cr
  &=\pm\frac{n!}{ r!}\cdot\frac{1}{\displaystyle\prod_{h\text{ hook of }S}
   (d\,l(h)+c_{i(h)}-c_{j(h)})}
\end{aligned}$$
as claimed. \par
Now choose $q$ such that $q\equiv1\pmod p$. Then we have
$$\gamma(1)/\lambda(1)\equiv \pm\psi_\nu(1)\pmod{\Phi_d(q)}$$
by Theorem~\ref{thm:HLL}, and
$$\gamma(1)/\lambda(1)\equiv \pm\chi_\pi(1)/\chi_\mu(1)
  \pmod{q-1}$$
by our observation above. As $p$ divides both $\Phi_d(q)=q^{d-1}+\ldots+1$
and $q-1$, the stated congruence follows.
\end{proof}

Let's note the following special case of $p$-quotients $(\nu_1,\ldots,\nu_p)
\vdash w$
such that the corresponding $p$-symbol $S$ has $S_i=(w)$, $S_j=(0)$ for
$j\ne i$. Since these correspond to linear characters of the relative Weyl
group in $\GL_n(q)$, they parametrize characters of height~0 in $B$ by the
congruence in Theorem~\ref{thm:hookformula}. We obtain
\begin{equation} \label{eq:lincharSn}
\chi_\pi(1)=\frac{n!}{ p^wr!\,w!}\cdot
   \prod_{k=0}^{w-1}\prod_{j\ne i}|pk+c_i-c_j|^{-1} \cdot\chi_\mu(1).
\end{equation}

A $p$-blocks $B$ of $\fS_n$ labelled by a $p$-core $\mu\vdash n-wp$ is said
to be of weight $w$. So $w$ denotes the number of $p$-hooks which must be
removed from any partition $\pi$ indexing a character in $B$ to obtain
its core $\mu$. The block is said to be self-dual if $\mu$ is a self-dual
partition.

\begin{prop}   \label{prop:sym}
 Let $G=\fS_n$, $n\ge5$, $p$ a prime, and $B$ a $p$-block of $G$. Then one
 of the following occurs:
 \begin{itemize}
  \item[\rm(a)] $B$ is of weight (and hence defect)~0,
  \item[\rm(b)] $p=2$ and $B$ is of weight~1,
  \item[\rm(c)] $p=3$, $B$ is of weight~1 and self-dual, or
  \item[\rm(d)] $B$ contains two height~0 characters of different degrees
   $d_1<d_2$, either both indexed by non-self-dual partitions or with
   $d_2\ne 2d_1$.
 \end{itemize}
\end{prop}

\begin{proof}
We use the relative hook formula in~(\ref{eq:lincharSn}) for the character
degrees of $\fS_n$ for certain height~0 characters in $B$. We may assume that
the weight~$w$ of $B$ is positive. Let $\mu\vdash n-pw$ denote the $p$-core
associated to $B$, let $0=e_1<e_2\ldots<e_p$ be the ordered set
of the $c_i$ as in Theorem~\ref{thm:hookformula}, and
$f_i:=\prod_{k=0}^{w-1}\prod_{j\ne i}|pk+e_i-e_j|$, for $1\le i\le p$.
Note that by \cite[Prop.~3.5]{Ol} none of the partitions $\lambda_i$
corresponding to the $p$-quotients $S_i$ is self-dual, unless $w=1$ in
which case at most one of them is.
Clearly, $f_p>f_{p-1}$ unless $p=2$ and $w=1$ (which is case~(b)), which yields
two distinct height~0 degrees $d_1,d_2$. If both corresponding partitions are
self-dual, then $w=1$. But by Theorem~\ref{thm:hookformula} we have
$d_i\equiv\pm1\pmod p$, and then $d_1=d_2/2$ implies that $p=3$.
\end{proof}

\begin{cor}   \label{cor:alter}
 Let $p$ be a prime, $B$ a $p$-block of $\fA_n$. Then one of the following
 holds:
 \begin{itemize}
  \item[\rm(a)] $B$ is of defect~0,
  \item[\rm(b)] $p=3$, $B$ is of weight~1 (hence with cyclic defect group
   $C_3$), self-dual, and all $\chi\in\Irr(B)$ have the same degree, or
  \item[\rm(c)] $B$ contains two height~0 characters of different degrees.
 \end{itemize}
 In particular, the assertion of Theorem~\ref{main} holds when $S$ is
 an alternating group.
\end{cor}

\begin{proof}
Let $B$ be a $p$-block of $\fS_n$, containing all characters $\chi_\lambda$
for which $\lambda$ has fixed $p$-core $\mu\vdash (n-pw)$. According to
\cite[Prop.~12.2]{Ol}, for example, if $w>0$ then $B$ covers a unique block
$B_1$ of $\fA_n$. First assume that $p$ is odd. Let $\chi_1,\chi_2\in B$ be
two height zero characters of different degrees, parametrized by non self-dual
partitions, according to Proposition~\ref{prop:sym}. These restrict irreducibly
to characters of $\fA_n$ in $B_1$ of height~0. Similarly, if
$\chi_1,\chi_2\in B$ have different degrees $d_1<d_2$ with $d_2\ne 2d_1$, then
the restrictions of $\chi_1,\chi_2$ to $\fA_n$ contain characters of $B_1$ of
height~0 and of different degrees. \par
If $p=3$, $B$ is of weight~1 and self-dual, then two characters of $B$
have the same irreducible restriction and one splits into two constituents
for $\fA_n$. We obtain a block $B_1$ with defect group of order~3 and three
equal character degrees.
\par
For $p=2$, restriction of characters from $G$ to $B_1$ either preserves
heights or decreases it, by \cite[Prop.~12.5]{Ol}. Thus, we may conclude by
Proposition~\ref{prop:sym} unless $w=1$. Here, the two irreducible characters
in $B$ have the same restriction to $\fA_n$, so $B_1$ is a block with a unique
ordinary character, that is, a block of defect zero.
\end{proof}

Note that case~(b) of Corollary~\ref{cor:alter} occurs if and only if there
is a self-dual 3-core for $n-3$. The conditions for this to occur have been
worked out in \cite[Lemma~3.1]{AE09}.

It can be checked from the known character tables that the assertion of
Theorem~\ref{main} remain true for the faithful blocks of $2.\fA_n$ when
$n\le13$.

We complete our investigation of blocks of quasisimple groups by showing:

\begin{prop}   \label{prop:spor}
 The assertion of Theorem~\ref{main} holds when $S$ is sporadic or a simple
 group of Lie type with exceptional Schur multiplier, or $S=\tw2F_4(2)'$.
\end{prop}

\begin{proof}
The ordinary character tables of all quasi-simple groups such that $S$ is
as in the assumption are contained in the Atlas \cite{Atl}. From this, or
using the electronic tables available in {\sf GAP}, it can be checked that
whenever $B$ is a $p$-block of $G$ with all height zero characters of the
same degree then the defect group satisfies $|D|\le p^2$, hence must be
abelian.
\end{proof}



\section{$p$-Solvable Groups}    \label{psolv}
Our main result in this section is to reduce the study of
EHZD blocks of general $p$-solvable groups to
groups with $p'$-lenght one. This latter case naturally leads us to consider
a variation of a classical {\sl large orbit} problem. 

\medskip

\begin{question}\label{question}
Suppose that $V$ is a finite faithful completely reducible
$FG$-module, where $F$ has characteristic $p$ and
$G$ has a normal $p$-complement $K>1$. Let $P \in \syl pG$. 
 Does there exists $v \in \cent VP$ such that 
$|\cent Kv|^2 < |K|$?
\end{question}

\medskip
 
Question \ref{question} is not trivial, even if $P=1$. In this case,
it has an affirmative answer if $K$ is solvable
(by \cite{D}). Also, Question ~~\ref{question}
 has an affirmative answer if $K$ is nilpotent, and this constitutes
  the main
result of \cite{D-N}.  In some sense, it is unfortunate that
 our only way to prove that EHZD blocks of $p$-solvable groups
are nilpotent is via large orbits. On the other hand,
Question ~~\ref{question} has interest in its own and it is closely
related to the study of $p'$-degrees of $p$-solvable
groups, so it might deserve some consideration.
\medskip

Our main result in this Section is the following.

\medskip

\begin{thm}\label{mainpsolvable}
Suppose that Question~~\ref{question} has an affirmative answer. If $B$
 is an EHZD block of a $p$-solvable group $G$, then $B$ is nilpotent.
\end{thm}

If $\lambda \in \irr N$ is a character,
we write $o(\lambda)$ for its determinantal order. If $N\nor G$,
then recall that  $\irr{G|\lambda}$ are the irreducible characters of $G$
lying over $\lambda$.

\begin{lem}\label{lempsolvable}
 Let $Z \nor G$ and suppose $\lambda \in \irr Z$ is $G$-invariant and that
 $o(\lambda)\lambda(1)$ is a $p'$-number. Assume that $G/Z$ is $p$-solvable and
 that $\oh{p'}{G/Z} = 1$.  Suppose that Question ~~\ref{question}
 has an affirmative answer. If there exists
 an integer $d$ such that $\chi(1)=d$ for all $\chi \in \irr{G|\lambda}$ 
 of $p'$-degree, then $G/Z$ is a $p$-group.
\end{lem}

\begin{proof}
We argue  by induction on $|G/Z|$. We may certainly assume that
$|G/Z| > 1$.  By using character triple isomorphisms 
(see Theorem (3.1) of \cite{N2}), we may also assume that
$Z$ is a central $p'$-group. Thus $Z=\oh {p'}G$.

Let $U/Z = \oh{p}{G/Z}$ and note that $U > Z$. We suppose that $U < G$,
and we seek a contradiction.
Let $K/U = \oh{p'}{G/U}$. Note that $K > U$. Also, write $U=V\times Z$,
and note that $U=\oh pG$. Also, by Hall-Higman's 1.2.3. Lemma,
we have that $\cent GV \sbs U \times Z$.
If $V_1=\Phi(V)$, by elementary group theory we have that
$\oh{p'}{G/ZV_1}=1$. If $\tilde\lambda=1_{V_1} \times \lambda \in \irr{V_1 \times Z}$,
then $\irr{G|\tilde\lambda} \sbs \irr{G|\lambda}$,
and by induction we will conclude that
$G/ZV_1$ is a $p$-group, and this will prove the theorem.
So we may assume that $V=\oh pG$ is elementary abelian.
Hence $\cent GV=U \times Z$.

Now let $K_0 = \Oh pK$ and $U_0 = U \cap K_0$. Note that $Z \sbs U_0$ and that 
$K_0/U_0 = \oh{p'}{G/U_0}$. In particular,
 $\oh{p'}{G/K_0}$ is trivial. Also, for all characters
$\phi \in \irr{K_0|\lambda}$, we have $o(\phi)$ and $\phi(1)$ are
${p'}$-numbers (because $K_0$ has a normal
abelian Sylow $p$-subgroup and $\Oh p{K_0}=K_0$).

Now fix $P \in \syl pG$ and suppose that $\phi \in  \irr{K_0|\lambda}$ is $P$-invariant. Write $T = G_\phi$ for the stabilizer
of $\phi$ in $G$. Hence $|G:T|$ is not divisible
by $p$. 
We claim that $T$ satisfies the hypotheses of
the theorem with respect to the character $\phi$ and the normal
subgroup $K_0 \nor T$. Notice that all $p'$-degree
members $\psi$ of $\irr{T|\phi}$ induce   to $p'$-degree
characters of $G$,  therefore
of degree $d$. Hence $\psi(1)=d/|G:T|$. Hence to prove the claim, we need to check that $\oh{p'}{T/K_0}$ is trivial.
Let $W/K_0 = \oh{p}{G/K_0} \sbs PK_0/K_0$,
and therefore $W$ stabilizes
$\phi$.   Thus $W \sbs T$ and $\oh{p'}{T/K_0}$
centralizes the normal $p$-subgroup $W/K_0$.
But $\oh{p'}{G/K_0}$ is trivial, and Hall-Higman's Lemma~1.2.3 applies to show
that $\oh{p'}{T/K_0} = 1$, as wanted.

By the inductive hypothesis, we conclude that $T/K_0$ is a
$p$-group, hence $T=K_0P$ (since $P \in \syl pT$).  We have proved this
for all $P$-invariant $\phi \in  \irr{K_0|\lambda}$.

Now let $G_0=PK_0$. We claim that $G_0$
satisfies the hypothesis of the theorem
with respect to $Z\nor G_0$. Let $\tau \in {\rm Irr}_{p'}(G_0|\lambda)$
and let $\phi=\tau_{K_0} \in \irr{K_0|\lambda}$ which is  $P$-invariant and irreducible
(because $G_0/K_0$ is a $p$-group).
We know that $G_0$ is the stabilizer in $G$ of $\phi$,
by the previous paragraphs. Therefore $\tau^G=\chi \in \irr G$
is irreducible of $p'$-degree.
Then $\chi(1)=d$,
and we conclude that $\tau(1)=d/|G:G_0|$. So
in order to prove the claim we just need
to show that $\oh{p'}{G_0/Z}=1$.
 However,  we have that $\oh{p'}{G_0/Z}$ centralizes $U/Z=\oh p{G/Z}$.
 Since $\oh {p'}{G/Z}=1$, we conclude that $\oh{p'}{G_0/Z}=1$.
  If $G_0 < G$,
the inductive hypothesis yields that $G_0/Z$ is a ${p}$-group, which contradicts
the fact that $K>U$. Hence we have that $G=PK_0$.
Thus $\bar G=G/U$ has a normal $p$-complement $\bar K=K/U$.

Now we have that $\irr V$ is a completely reducible, finite, and faithful
$\bar G$-module. By using the affirmative answer to
Question (\ref{question}), there exists $\beta \in \irr V$
centralized by $P$ such that 
$$|K_\beta/U|^2 < |K/U| \, ,$$
where $K_\beta$ is the stabilizer in $K$ of $\beta$.
In other words,
$$|K:K_\beta|^2>|K/U| \, .$$

Now, since $K_\beta/V$ is a $p'$-group,
there exists a unique extension $\hat\beta \in \irr{K_\beta}$
of $\beta$, by using Corollary (8.16) of \cite{I}, which has $p$-power order. 
In particular, this linear character has $Z$ in its kernel, and by uniqueness is $P$-invariant
(because $\beta$ is $P$-invariant).
Let $\hat\lambda=1_V \times \lambda \in \irr U$.
Since $K_\beta/U$ is a $p'$-group
and $\hat\lambda$ is $P$-invariant, then
we may find some $\gamma \in \irr{K_\beta| \hat\lambda}$ which
is $P$-invariant (this is because $\hat\lambda^{K_\beta}$ has $p'$-degree).
 Now we have that
  $\gamma\hat\beta \in \irr{K_\beta}$ (because $\hat\beta$ is linear)  lies over $\beta$.
  By the Clifford correspondence, we have that 
$\rho=(\gamma \hat\beta)^K \in \irr K$. This character $\rho$
 is $P$-invariant, has $p'$-degree $|K:K_\beta|\gamma(1)$ and lies over $\lambda$.
Also $\rho_{K_0} \in \irr{K_0}$ is $P$-invariant,
has $p'$-degree, and therefore it has an 
 extension $\chi \in \irr G$ with $\chi_{K_0}=\rho_{K_0}$, by using Corollary (8.16) of \cite{I} and the fact
 that $K_0=\Oh pG$.
 Hence $$d=|K:K_\beta|\gamma(1) \ge |K:K_\beta | \, .$$
Therefore,
 $$d^2 \ge |K:K_\beta |^2 > |K/U| \, .$$
 
 Now, let $H$ be a $p$-complement of $G$. Hence $HV=K$
 and $H\cap V=1$.

  Finally, using that  $\hat\lambda$ is $P$-invariant
  and $(\hat\lambda)^K$ has $p'$-degree, 
  we can find a $P$-invariant $\xi \in \irr{K|\hat\lambda}$ of $p'$-degree. Arguing as before,
  we have that $\xi$ extends to $G$, and therefore
  $\xi(1)=d$. However, $\xi_H \in \irr{H|\lambda}$. 
  Hence, $d^2 \le |H:Z|$ by elementary character theory.
  However, $|H:Z|=|K:U|$ and   this is a contradiction.  
 \end{proof}
 
 \medskip
(A similar argument gives the same conclusion of Lemma 11.3 if
we assume that $\chi^0 \in \ibr G$ for all $\chi \in \irr{G|\lambda}$ of $p'$-degree.)

\medskip

\iitem{Proof of Theorem~~\ref{mainpsolvable}.}~~We argue by induction
on $|G|$. Let $Z=\oh {p'}G$, and let $\lambda \in \irr Z$ be
covered by $B$. If $T$ is the stabilizer of $\lambda$ in $G$ and $b$ is the block
of $T$ which corresponds to $B$ via Fong-Reynolds (\cite{N}, Theorem (9.14)), then $b$ is a EHZD block.
If $T<G$, then $b$ is nilpotent by induction. Thus $B$ is nilpotent by Lemma 1 of \cite{N1}, for instance.
Hence, we may assume that $T=G$. In this case, $\irr B= \irr{G|\lambda}$ by Theorem (10.20) of \cite{N}.
Now we conclude that $G$ has a normal $p$-complement by Lemma \ref{lempsolvable}. \qed

\bigskip


\end{document}